\documentclass[a4paper]{amsart}

\usepackage{latexsym,amssymb,amsmath,amsthm,amsfonts}

\usepackage{hyperref}

\usepackage[dvipsnames]{xcolor}

\hypersetup{
    colorlinks = true,
    linkcolor = {BrickRed},
    citecolor = {NavyBlue}
}
\usepackage{nicematrix}

\newcommand{\kk}{\Bbbk}

\newcommand{\upto}{,\ldots ,}

\newcommand{\betafield}{\beta_{\operatorname{field}}}
\newcommand{\Vreg}{V_{\operatorname{reg}}}

\DeclareMathOperator{\charakt}{char}

\def\SL{\operatorname{SL}}

\def\SL2{\operatorname{SL}_{2}(K)}

\def\GL2{\operatorname{GL}_{2}(K)}

\def\INVSL2{$K[V]^{operatorname{SL}_{2}(K)}$}
\def\INVSO2{$K[V]^{operatorname{SO}_{2}(K)}$}
\def\INVGL2{$K[V]^{operatorname{GL}_{2}(K)}$}

\def\GL{\operatorname{GL}}
\def\SL{\operatorname{SL}}
\def\id{\operatorname{id}}

\def\Z{\mathbb{Z}}

\newtheorem{Lemma}{Lemma}
\newtheorem{Theorem}{Theorem}
\newtheorem{Proposition}{Proposition}

\theoremstyle{definition}
  \newtheorem{Def}{Definition}

\theoremstyle{remark}
  \newtheorem{rem}{Remark}

\newtheoremstyle{Acknowledgments}
  {}
    {}
     {}
     {}
    {\bfseries}
    {}
     {.5em}
     {\thmname{#1}\thmnumber{ }\thmnote{ (#3)}}
\theoremstyle{Acknowledgments}

\numberwithin{equation}{section}

\title[Generic separation for modular invariants]{ Generic separation for modular invariants }

\author{Fabian Reimers}
 \address{Technische Universit\"at M\"unchen, Department of Mathematics, 
Boltzmannstr.~3, 85748 Garching, Germany}
\email{reimers@ma.tum.de}

\author{M\"{u}fit Sezer}
\address{Bilkent University, Department of Mathematics\\
Cankaya, Ankara \\06800 Turkey } \email{sezer@fen.bilkent.edu.tr}

\thanks{We thank    T\"UBITAK for funding a visit of the first author to Bilkent University}
\date{\today}

\begin{document}
\begin{abstract} For modular indecomposable representations of a cyclic group~$G$ of prime order~$p$ we propose a list of polynomial invariants of degree~$\leq 3$ that, together with a simple invariant of degree $p$, separate generic orbits and generate the field of rational invariants. A similar result is proven for decomposable representations of~$G$. 
 \end{abstract}
\maketitle  

\section{Introduction}     
  Let $G$ be a finite group, and let $V$ be a finite-dimensional representation of~$G$ over a field~$\kk$. Invariant theory considers the induced action of~$G$ on the symmetric algebra~$\kk[V]$ on the dual space~$V^*$. With a basis $x_1 \upto x_n$ of $V^*$, the algebra $\kk[V] = \kk[x_1 \upto x_n]$ can be viewed as a polynomial ring with $x_1 \upto x_n$ as indeterminates, and the invariant ring is \[ \kk[V]^G = \{ f \in \kk[V] \mid g \cdot f = f \; \; \forall  g \in G  \} .\]
Typical problems are to find \emph{generating} invariants for $\kk[V]^G$ as a $\kk$-algebra, or to find \emph{separating} invariants, which by definition separate 
$G$-orbits in~$V$. Since $\kk[V]^G$ inherits the grading of the polynomial ring $\kk[V]$, one is also interested in obtaining degree bounds on generating or separating invariants such as  the well-known  Noether bound.

Motivated by applications in signal processing, a program initiated by Bandeira et al. \cite{BANDEIRA2023236} and Blum-Smith et al. \cite{BLUMSMITH2024107693} shifts the focus towards the invariant field $\kk(V)^G$ inside the rational function field $\kk(V)$. Again we are interested in explicit generators of $\kk(V)^G$, as a field extension of $\kk$, or in degree bounds on generators. Closely related, is the task of finding invariants that are \emph{generically separating}, which is defined as follows. Given a $G$-stable subset $B \subseteq V$, a set $S$ of invariant polynomials is said to be $B$-separating if for any two orbits $Gv \neq Gw$ with $v,w \in B$ there exists a polynomial $f \in S$ with $f(v) \neq f(w)$. A set $S\subseteq \kk[V]^G$ is said to be \emph{generically separating} if it is $B$-separating for some $G$-stable, nonempty, Zariski open set $B \subseteq V$. 

Let $\betafield(G,V)$ denote the smallest number $d$ such that the invariant field is generated by polynomial invariants of degree $\leq d$. Fleischmann, Kemper and Woodcock~\cite[Corollary 2.3]{FLEISCHMANN2007497} proved that $\betafield(G,V) \leq |G|$ holds independently of the characteristic of $\kk$. For the regular representation~$\Vreg$ of a finite abelian group~$G$ over a field~$\kk$ with $\charakt(\kk) \nmid |G|$, Bandeira et al.~\cite[Theorem 4.1 and Remark 4.2]{BANDEIRA2023236} showed that $\betafield(G,\Vreg) \leq 3$. Recently, this result about~$\Vreg$ was extended by Edidin and Katz~\cite[Theorem III.1]{edidin2025orbitrecoveryinvariantslow} to any finite group over any infinite field. They also showed in~\cite[Theorem 2.5]{edidin2024generic} that for the $n$-dimensional standard representation of a dihedral group $D_n$ over the complex numbers, the same upper bound of~$\betafield(D_n,V) \leq 3$ holds. On the other hand, Blum-Smith et al. \cite[Theorem 3.2]{BLUMSMITH2024107693} proved as a lower bound that we need polynomial invariants of degree at least~$\sqrt[n]{|G|}$ with $n = \dim(V)$ to generate the invariant field.
 

 In this paper we study these questions in the modular case, i.e., when the characteristic of $\kk$ divides the order of $G$. Compared to the non-modular case, the invariant ring tends to be more complicated and the degrees of the generators grow fast. Even in the basic example of 
a cyclic group of prime order~$p$, explicit  generators for the invariants are only known for indecomposable representations of dimension at most~5 (see Shank~\cite{Shankv4v5}) and a limited number of decomposable representations (see Wehlau~\cite{MR3085091} and the references there). On the other hand, it is possible to construct a separating set of invariants for the $n$-dimensional indecomposable representation from one for the $(n-1)$-dimensional representation (see Sezer~\cite{SezerSeparating}). In addition, a generating set for the invariant field of this group has been computed for any modular representation by Campbell and Chuai~\cite[Theorem 3.1]{MR2334859}. For the $n$-dimensional indecomposable representation their generating set contains  polynomials of each degree from one to $n-1$ and a polynomial of degree $p$. Our main result is a strengthening of this generating set. 
In Theorem~\ref{MainTheoIndecomp} we construct a generically separating set (which is also a generating set for the invariant rational functions) for the $n$-dimensional indecomposable representation consisting of~$n$  invariants, which are of degree~$\leq 3$ except for one invariant of degree~$p$.

This paper is structured as follows. In Section~\ref{SectionGenericallySep} we gather some basic results about generically separating invariants. Then we specialize to modular representations of a cyclic group of prime order. In Subsection~\ref{SubSecDeg2}, for $n$ odd, we find an invariant of degree 2 that lifts a generically separating set for the $(n-1)$-dimensional indecomposable representation to the $n$-dimensional indecomposable representation. With more computational effort, but essentially with the same ideas, we find an invariant of degree 3 that does the same job for $n$ even in Subsection~\ref{SubSecDeg3}. Finally, Subsection~\ref{SubSecMainResults} contains the main results, Theorem~\ref{MainTheoIndecomp} and Theorem~\ref{MainTheoDecomp}, and compares the invariants we found with those in the literature.

\section{Basics about generically separating invariants}\label{SectionGenericallySep}

Polynomial invariants that generate the invariant field also generate a simple localization of the invariant ring (as has been observed, e.g. in \cite{MR2334859} and in \cite{FLEISCHMANN2007497}). From this it follows easily that they are generically separating as well. For the convenience of the reader, we prove this statement here.

Recall that throughout this paper, $G$ is a finite group and $V$ is a finite-dimensional representation of~$G$ over a field~$\kk$.

\begin{Proposition}\label{PropGenSeparatingAndFieldGen}
Consider the following properties for a list of polynomial invariants $f_1 \upto f_r \in \kk[V]^G$:
\begin{enumerate}
    \item[(1)] $f_1 \upto f_r$ generate the invariant field as a field extension of $\kk$, i.e., $\kk(V)^G = \kk(f_1 \upto f_r)$,
    \item[(2)] there exists an invariant $g \in \kk[f_1 \upto f_r] \setminus \{ 0 \}$ such that the localized ring $\kk[V]^G_g := \{ f / g^a \mid f \in \kk[V]^G, a \geq 0 \}$ is generated as a $\kk$-algebra by $f_1 \upto f_r, g^{-1}$,
    \item[(3)] $f_1 \upto f_r$ are generically separating.
\end{enumerate}
Then (1) and (2) are equivalent, and, if $\kk$ is infinite, (1) implies (3).
\end{Proposition}

\begin{proof}
First, let us assume (1) and show (2). Take generators $p_1 \upto p_s$ of the invariant ring as a $\kk$-algebra. By assumption, they are contained in the field generated by $f_1 \upto f_r$. Using a common denominator, we get a polynomial $g \in \kk[f_1 \upto f_r] \setminus \{ 0 \}$ and for each $i$ an invariant $h_i \in \kk[f_1 \upto f_r]$  such that $p_i = h_i / g$. Hence, we have $\kk[V]^G \subseteq \kk[f_1 \upto f_r, g^{-1}]$ and the result follows.

Conversely, if we assume (2), then the invariant ring is contained in the field generated by $f_1 \upto f_r$. As $\kk(V)^G$ is the field of fractions of $\kk[V]^G$ (see \cite[Lemma 3.11.3]{MR3445218}), this implies (1).

Now, let us assume (2) and that $\kk$ is infinite, and show (3). Define the subset \[ B := \{ v \in V \mid g(v) \neq 0 \}.\] This set is Zariski open (since $g$ is a polynomial), $G$-stable (since $g$ is an invariant) and nonempty (since $g \neq 0$ and $\kk$ is infinite). Let $v,w \in B$ be such that $f_i(v) = f_i(w)$ for $i = 1 \upto r$. Since $g \in \kk[f_1 \upto f_r]$, we get $g(v) = g(w)$. By assumption, any invariant $f \in \kk[V]^G$ can be written as $f = h / g^a$ with $h \in \kk[f_1 \upto f_r]$ and $a \geq 0$. It follows that $f(v) = f(w)$ for all invariants, and hence $v$ and $w$ are in the same $G$-orbit.
\end{proof}

If $\kk$ is algebraically closed of characteristic 0, then it is known that all conditions in Proposition~\ref{PropGenSeparatingAndFieldGen} are equivalent, as (3) implies (1) by \cite[Lemma 2.1]{popov1994invariant}. 

In characteristic $p > 0$, however, there is an obvious gap between conditions~(3) and~(1) in Proposition~\ref{PropGenSeparatingAndFieldGen}: A polynomial and its $p$-th power share the same separation capabilities, but generate different algebraic objects. 

\bigskip

Our construction of a set of generically separating invariants of low degree in Theorems~\ref{TheoremDegree2andNodd}~and~\ref{PropositionDegree3AndNeven} are based heavily on the following observation, which itself is a generalization of \cite[Theorem 1.1]{MR3078070} from \emph{separating} to \emph{$B$-separating}, as defined in the introduction.

  \begin{Proposition}\label{PropInductiveSeparationConstr}
  Let $V$, $W$ be representations of $G$ together with a $G$-equi\-variant surjection $\phi: V\rightarrow W$ and let  $\phi^*: \kk[W]\rightarrow \kk[V]$ be the corresponding inclusion. Let~$B$ be a $G$-stable subset of $V$. Furthermore, let $S\subseteq \kk[W]^G $ be a $\phi(B)$-separating set, and let $T \subseteq \kk[V]^G$ be a set with the property that for different 
   $G$-orbits $Gv_1 \neq G v_2$ with $v_1, v_2\in B$ and $\phi (v_1)=\phi (v_2)$, there exists a polynomial $f\in T$ such that $f(v_1)\neq f(v_2)$. Then $\phi^*(S) \cup T$ is $B$-separating. 
  \end{Proposition}
  \begin{proof}
  Let $v_1$ and $v_2$ be two vectors in $B$ in different $G$-orbits. If $\phi (v_1)$ and $\phi (v_2)$ are in the same $G$-orbit in $\phi (B)$, then there exists a group element $g \in G$ such that $\phi(g v_1) = g \phi(v_1) = \phi(v_2)$. By assumption, there is an invariant $f \in T$ such that $f(v_1) = f(g v_1) \neq f(v_2)$.
If $\phi(v_1)$ and $\phi(v_2)$ are in different $G$-orbits, then, by assumption, there exists an invariant $f \in S$ with $f(\phi(v_1)) \neq  f(\phi(v_2))$, and hence $\phi^*(f)$ separates $v_1$ and $v_2$. 
 \end{proof}

\section{Cyclic group of prime order}\label{SectionCyclicGroup} From here on, let $G$ denote a cyclic group of prime order~$p$, and $\kk$ a field of characteristic $p > 0$. We fix a generator~$\sigma$ of~$G$. It is well known that up to isomorphism there are exactly~$p$ indecomposable representations $V_1, V_2 \upto V_p$ of~$G$, where~$\sigma$ acts
on~$V_n$ for $ 1 \le  n \le  p$ by a Jordan block of dimension~$n$ with $1$'s on the diagonal (see, e.g., \cite[Lemma 7.1.3]{campbell2011modular}). Let $e_1, e_2 \upto e_n$ be a Jordan block basis for~$V_n$ with $\sigma (e_i) = e_i + e_{i+1}$ for $1\le   i \le  n-1$ and $\sigma (e_n) = e_n$ and let $x_1, x_2 \upto x_n$ denote the corresponding dual basis, so that
\[ \kk [V_n] = \kk [x_1, x_2,\dots , x_n].\]
Since $V_n^*$ is indecomposable it is isomorphic to $V_n$. In fact, $x_1, x_2 \upto x_n$ form a Jordan block basis for $V_n^*$ in reverse order. For convenience, we replace $\sigma$ with $\sigma^{-1}$, so that we have 
\[ \sigma (x_i) = x_{i-1} + x_i  \quad \text{ for } 2\le i \le  n \quad \text{and } \sigma (x_1) = x_1.\] 
 We identify the vectors of $V_n$ with column vectors with respect to the ordered basis $e_1 \upto e_n$. There is a $G$-equivariant surjection\begin{equation}\label{DefPhi}
\phi_n: V_n \rightarrow  V_{n-1}\, \quad  (c_1, c_2 \upto c_n)^t\rightarrow  (c_1, c_2 \upto c_{n-1})^t, 
\end{equation}
which corresponds to the inclusion $\kk[V_{n-1}] \subseteq \kk[V_n]$. Thus, there is an inclusion of invariant rings $\kk[V_{n-1}]^G \subseteq \kk[V_n]^G$ as well.

  We will work with the Zariski open set 
  \begin{equation}\label{DefBn} B_n:=\{ (c_1, c_2,..., c_n)^t \mid c_1\neq 0 \} \subseteq V_n.\end{equation} Note that $B_n$ is $G$-stable and with $\phi_n$ defined in~\eqref{DefPhi} we have $\phi (B_n)=B_{n-1}$. Our aim is to inductively construct a $B_n$-separating set in 
  $\kk [V_{n}]^G$.

\subsection{Degree two connecting invariant}\label{SubSecDeg2}

Let $W$ be the vector space in $\kk [V_n]$ spanned by monomials of degree 2. For a monomial $x_i x_j\in W$, we call $d = i+j$ its \emph{weight}. We will use the decomposition $W = \bigoplus_{d \geq 0} W_d$, where~$W_d$ denotes the subspace spanned by all degree 2 monomials of weight $d$. 

Define the operator 
\[ \Delta: \kk[V_n] \to \kk[V_n], \quad \Delta := \sigma - \id,\] whose kernel is the invariant ring. Note that for $i,j \geq 2$ we have
\[ \Delta(x_i x_j) = x_{i-1} x_{j-1} + x_{i-1} x_j + x_{i} x_{j-1}, \]
and hence for the weight spaces we get
\begin{equation}\label{DeltaWd} \Delta(W_d) \subseteq W_{d-2} + W_{d-1},
\end{equation}
which we will refine in the following two lemmas.  As a polynomial $f\in W$ decomposes as $f=\sum_df_d$ with $f_d\in W_d$, for a nonnegative integer $d$, let $\Delta_d$ be the linear map on $W$ defined by $\Delta_d(f) = (\Delta(f))_d$. 


  \begin{Lemma}
  \label{LemmaDeg2andOdd}
      Let $d\le n+1$ be an odd integer. Then we have
      $$\Delta_{d-1}(W_d) = W_{d-1},$$
      in fact, the restriction $\Delta_{d-1}: W_d \to W_{d-1}$ is an isomorphism.
  \end{Lemma}

  \begin{proof}
   Note that $W_0=W_1=\{0\}$, so we may take $d \geq 3$.  Consider the ordered bases $\{x_1x_{d-1},\, x_2x_{d-2} \upto x_{\frac{d-1}{2}} x_{\frac{d+1}{2}}\}$ of $W_d$ and $\{ x_1 x_{d-2},\, x_2x_{d-3} \upto x_{\frac{d-1}{2}}^2\}$ of $W_{d-1}$. Note that $\Delta(x_1 x_{d-1}) = x_1 x_{d-2}$ and for $2 \leq i < j$ with $i+j = d$ we have $\Delta_{d-1}(x_i x_j) = x_{i-1} x_j + x_i x_{j-1}$. It follows that the restriction $\Delta_{d-1}: W_d \to W_{d-1}$ is represented, with respect to these ordered bases, by an upper Jordan block with~$1$'s along the diagonal. Therefore, the result follows. 
    \end{proof}
    
 For $d$ even, we will need the following strengthening of Lemma~\ref{LemmaDeg2andOdd}.

\begin{Lemma}\label{LemmaDeg2andEven}
      Let $d\le n+1$ be an even integer and let $W_d'$ denote the subspace of $W_d$ generated by all monomials in $W_d$ except $x_1 x_{d-1}$. Then we have 
      $$\Delta_{d-1}(W'_d) = W_{d-1},$$
       in fact, the restriction $\Delta_{d-1}: W_d' \to W_{d-1}$ is an isomorphism.
  \end{Lemma}
 \begin{proof}
 Note that $W_1= W'_2= \{0\}$, so we may take $d\ge 4$. Consider the ordered bases $\{x_2x_{d-2},\, x_3 x_{d-3} \upto x^2_{\frac{d}{2}}\}$ of $W_d'$ and $\{x_1 x_{d-2},\, x_2x_{d-3} \upto x_{\frac{d}{2}-1} x_{\frac{d}{2}}\}$ of $W_{d-1}$. 
   For $2 \leq i < j$ with $i+j=d$ with we have $\Delta_{d-1}(x_i x_j) = x_{i-1} x_j + x_i x_{j-1}$, and $\Delta_{d-1}(x^2_{\frac{d}{2}}) = 2 x_{\frac{d}{2}-1} x_{\frac{d}{2}}$. It follows that the restriction $\Delta_{d-1}: W_d' \to W_{d-1}$ is represented (with respect to these ordered bases) by a lower triangular matrix whose diagonal entries are~$1$'s and a single 2. This matrix has full rank, and therefore the result follows. 
\end{proof}
     
\begin{Proposition}\label{PropositionDegree2AndOdd} Let $n\ge 3$ be an odd integer. There exists a homogeneous invariant~$f_n\in \kk [V_{n}]^G$
  of degree 2 such that $f_n=x_1x_n+h$, where $h\in \kk [V_{n-1}]=\kk [x_1, x_2,\dots , x_{n-1}]$.
\end{Proposition}
  
\begin{proof}

We have $\Delta_n(x_1x_n) = \Delta(x_1x_n) = x_1 x_{n-1} \in W_n$. So by Lemma~\ref{LemmaDeg2andEven} with $d = n + 1$ there exists $g_{n+1} \in W_{n+1}'$ such that $\Delta_n(g_{n+1}) = \Delta_n(x_1 x_n)$ and so $\Delta(x_1 x_n - g_{n+1}) \in W_{n-1}$ by~\eqref{DeltaWd}. By Lemma~\ref{LemmaDeg2andOdd} with $d = n$ there exists $g_n \in W_n$ such that $\Delta_{n-1}(g_{n}) = \Delta_{n-1}(x_1 x_n - g_{n+1})$ and so $\Delta(x_1 x_n - g_{n+1}  - g_n) \in W_{n-2}$ again by~\eqref{DeltaWd}.

Continuing in this fashion, inductively, but in reverse order for $i = n+1 \upto 3$, we find polynomials~$g_i$ such that
\begin{itemize}
    \item $g_i \in W_i$ (and if $i$ is even, then $g_i \in W_i'$),
    \item $\Delta(x_1x_n - \sum_{j=i}^{n+1} g_j) \in W_{i-2}$.
\end{itemize}
In the end, define $h := - \sum_{j=3}^{n+1} g_j$. Then $\Delta(x_1 x_n + h) \in W_1 = \{ 0 \}$, and hence $x_1x_n + h$ is an invariant. In addition, since $h \in W_3 \oplus \ldots \oplus W_n \oplus W_{n+1}'$, we see that no monomial of $h$ is divisible by $x_n$, so $h \in \kk[x_1 \upto x_{n-1}]$.
\end{proof}

\begin{Def}\label{DefFNofDeg2} 
We can define  $f_n$ (for $n\ge 3$ odd) as the unique invariant produced by the proof of Proposition~\ref{PropositionDegree2AndOdd}, as we note that, whenever the proof uses the surjectivity of $\Delta_{d-1}$ in Lemmas~\ref{LemmaDeg2andOdd} and Lemma~\ref{LemmaDeg2andEven}, it is, in fact, an isomorphism. 
For $n = 3$ we get
\[ f_3 = x_1 x_3 - \frac{1}{2} x_2^2 + \frac{1}{2} x_1 x_2,\]
and for $n = 5$ we get
\[ f_5 = x_1 x_5 - x_2x_4 + \frac{1}{2}x_3^2 - \frac{1}{2} x_2x_3 + \frac{3}{2} x_1x_4 + \frac{1}{4} x_2^2 - \frac{1}{4} x_1 x_2.\]
This is independent of $p$ (provided $n \leq p$), see the comments at the end of the paper about \emph{integral invariants}.
\end{Def}
  
  \begin{Theorem}\label{TheoremDegree2andNodd} 
      Let $n\ge 3$ be an odd integer, and let $B_n$ be the open set defined in~\eqref{DefBn}. Furthermore, let $S \subseteq \kk [V_{n-1}]^G$ be a $B_{n-1}$-separating set. Then $S \cup \{f_n\}$ is a $B_n$-separating set, where $f_n$ is the homogeneous invariant   in $\kk [V_{n}]^G$ of degree~2  defined just before the theorem. 
  \end{Theorem}
  \begin{proof}
 Take $f_n = x_1 x_n + h \in \kk[V_n]^G$ with $h \in \kk[V_{n-1}]$ from Proposition~\ref{PropositionDegree2AndOdd}. In view of Proposition~\ref{PropInductiveSeparationConstr}, using $\phi_n$ as defined in~\eqref{DefPhi}, it is enough to show that~$f_n$ separates any two vectors~$v_1, v_2\in B_n$ in different $G$-orbits with  $\phi_n (v_1)=\phi_n (v_2)$. So we may assume that $v_1=(c_1, c_2 \upto c_{n-1} ,c_n)^t$ and  $v_2=(c_1, c_2 \upto c_{n-1}, d_n)^t$
 with $c_1\neq 0$ and $c_n \neq d_n$. Then 
 \[ f_n(v_1) = c_1 \underbrace{c_n}_{\neq \, d_n} + h(\underbrace{\phi (v_1)}_{=\, \phi(v_2)}) \neq c_1 d_n + h(\phi (v_2)) = f_n(v_2).\]
 \end{proof}

Proposition~\ref{PropositionDegree2AndOdd} no longer holds if $n$ is even, e.g., for $n = 4$ (and hence $p \geq  5$) the space of homogeneous invariants of degree 2 is spanned by $x_1^2$ and $f_3$ (see~\cite[Theorem 4.1]{Shankv4v5}). Therefore, for even $n$ we will look for an invariant in degree~3. 

 \subsection{Degree three connecting invariants}\label{SubSecDeg3}

To construct a connecting invariant of degree 3 it will suffice to consider monomials that contain the variables $x_1$ or~$x_2$. Therefore, let $S$ be the vector space in $\kk [V_n]$ spanned by monomials of degree~3 that are divisible by $x_1$ or by $x_2$, and for an integer $d$, let $S_d$ denote the subspace in $S$ spanned by all $x_i x_j x_k$ with weight $i+j+k = d$.

Before diving into computations, the reader familiar with the previous subsection might have a glance at Proposition~\ref{PropositionDegree3AndNeven} and Theorem~\ref{TheoremDegree3AndNeven}, the analogues of Proposition~\ref{PropositionDegree2AndOdd} and Theorem~\ref{TheoremDegree2andNodd}. 

Similarly to \eqref{DeltaWd}, we have
\begin{equation}\label{DeltaSd}
    \Delta(S_d) \subseteq S_{d-3} + S_{d-2} + S_{d-1},
\end{equation}
which again we will examine more closely.  A polynomial $f\in S$ decomposes as $f=\sum_d f_d$ with $f_d \in S_d$ and for a nonnegative integer~$d$, let~$\Delta_d$ be the linear map on~$S$ defined by $\Delta_d(f) = (\Delta(f))_d$.

\begin{Lemma}\label{LemmaDeg3andOdd}
Let $d \leq n+2$ be an odd integer or $d = 4$. Then we have
\[ \Delta_{d-1}(S_d) = S_{d-1}.\]
\end{Lemma}

\begin{proof}
For small $d$ we have
\[ S_1 = S_2 = \{ 0 \}, \quad S_3 = \langle x_1^3 \rangle, \quad S_4 = \langle x_1^2 x_2 \rangle, \quad S_5 = \langle x_1^2 x_3, \ x_1 x_2^2 \rangle. \]
So the assertion is trivial for $d=1,2$. Since $\Delta(x_1^3) = 0$, the statement is also true for $d = 3$. Since $\Delta(x_1^2 x_2) = \Delta_3(x_1^2 x_2) = x_1^3$, the statement is true for $d = 4$. Since $\Delta_4(x_1^2 x_3) = x_1^2 x_2$, the statement is true for $d = 5$ as well.

Now we assume that $d$ is odd and $d \geq 7$. Then we have the following ordered basis of $S_d$, split into monomials divisible by $x_1$ and those that are not:
\begin{enumerate}
    \item[(B1)] $x_1^2 x_{d-2},\ x_1 x_2 x_{d-3}, \ \ldots , \ x_1 x_{\frac{d-3}{2}} x_{\frac{d+1}{2}}, \ x_1 (x_{\frac{d-1}{2}})^2$, 
    \item[(B2)] $x_2^2 x_{d-4},\ x_2 x_3 x_{d-5}, \ \ldots , \ x_2 x_{\frac{d-3}{2}} x_{\frac{d-1}{2}}$.
\end{enumerate}
Similarly, we have the following ordered basis of $S_{d-1}$:
\begin{enumerate}
    \item[(C1)] $x_1^2 x_{d-3},\ x_1 x_2 x_{d-4}, \ \ldots , \ x_1 x_{\frac{d-3}{2}} x_{\frac{d-1}{2}}$, 
    \item[(C2)] $x_2^2 x_{d-5},\ x_2 x_3 x_{d-6}, \ \ldots , \ x_2 (x_{\frac{d-3}{2}})^2$.
\end{enumerate}
Observe that $\Delta_{d-1}(x_1^2 x_{d-2}) = x_1^2 x_{d-3}$ and for a monomial $x_1 x_i x_j$ of (B1) with $2 \leq i \leq j$ we have
\[ \Delta_{d-1}(x_1 x_i x_j) = x_1 x_{i-1} x_j + x_1 x_i x_{j-1}.\]
So the monomials of (B1) are mapped via $\Delta_{d-1}$ to the subspace generated by (C1). Similarly, we have $\Delta_{d-1}(x_2^2 x_{d-4}) = 2x_1x_2x_{d-4}+x_2^2 x_{d-5}$ and for a monomial $x_2 x_i x_j$ of (B2) with $3 \leq i < j$ we have 
\[ \Delta_{d-1}(x_2 x_i x_j) = x_1 x_i x_j + (x_2 x_{i-1} x_j + x_2 x_i x_{j-1}).\]
We can, in fact, omit the last vector of (B1) without hurting the surjectivity. Let $\widehat{S_d} \subseteq S_d$ be the subspace spanned by all basis vectors from (B1) and (B2) except the last one of (B1), and consider the restriction $\Delta_{d-1}: \widehat{S_d} \to S_{d-1}$. By the above computation, this surjection is represented with respect to these ordered bases by
\[
\begin{pNiceArray}{cccc|cccc}
 1&1&&&& \\[-1ex] &1 & \ddots &&&\Block{2-2}<\Huge>{\ast} \\[-1ex] && \ddots & 1 \\ &&& 1 \\ \hline &&&& 1&1  \\[-1ex] &\Block{2-2}<\Huge>{0}&&&&1 &\ddots \\[-1ex] &&&&&& \ddots & 1 \\ &&&&&&&1
\end{pNiceArray}.
\]
This square matrix is upper-triangular with 1's along the diagonal. Therefore, it has full rank, and hence the map $\Delta_{d-1}: S_d \to S_{d-1}$ is surjective.
\end{proof}

For $d$ even, we will need the following strengthening of Lemma~\ref{LemmaDeg3andOdd}.

\begin{Lemma}\label{LemmaDeg3andEven}
Let $d \leq n+2$ be an even integer with $d \geq 6$, and let $S_d'$ denote the subspace of $S_d$ generated by all monomials in $S_d$ except $x_1^2 x_{d-2}$. Then we have
\[ \Delta_{d-1}(S'_d) = S_{d-1}.\]
\end{Lemma}

\begin{proof}
Since $d$ is even, we have the following ordered basis of $S_d'$, split into monomials divisible by $x_1$ and those that are not:
\begin{enumerate}
    \item[(B1)] $x_1 x_2 x_{d-3},\ x_1 x_3 x_{d-4},\ \ldots,\ x_1 x_{\frac{d}{2}-1} x_{\frac{d}{2}}$ \qquad (these are $\frac{d}{2} - 2$ many),
    \item[(B2)] $x_2^2 x_{d-4}, \ x_2 x_3 x_{d-5},\ \ldots ,\ x_2 (x_{\frac{d}{2} - 1})^2$ \qquad (these are $\frac{d}{2} - 2$ many).
\end{enumerate}
Similarly, we have the following ordered basis of $S_{d-1}$:
\begin{enumerate}
    \item[(C1)] $x_1^2 x_{d-3},\ x_1 x_2 x_{d-4},\ \ldots,\ x_1 (x_{\frac{d}{2}-1})^2$ \qquad (these are $\frac{d}{2} - 1$ many),
    \item[(C2)] $x_2^2 x_{d-5},\ x_2 x_3 x_{d-6},\ \ldots,\ x_2 x_{\frac{d}{2} - 2} x_{\frac{d}{2} - 1}$ \qquad (these are $\frac{d}{2} - 3$ many),
\end{enumerate}
Both spaces have dimension $d-4$. A computation of $\Delta_{d-1}$ similar as in the proof of the previous lemma shows that restriction $\Delta_{d-1}: S_d' \to S_{d-1}$ with respect to these ordered bases is given by the following block matrix (not divided into squares at this point!):
\[ A = \begin{pNiceArray}{ccccc|ccccc}
    1 &&&&&0 &\\ 1 &1&&&&2 \\[-1.2ex] & 1& \ddots &&&& 1\\[-1ex] & & \ddots & 1 &&& & \ddots \\ &&& 1 &1&&&&1& \\  &&&  &1&&&&&1 \\ \hline &&&&&1&1 \\[-1ex] &&\Block{2-1}<\Huge>{0}&&&&1&\ddots \\[-1ex] &&&&&&&\ddots&1\\&&&&&&&&1&2
\end{pNiceArray}.\]  
From the $(\frac{d}{2}-2)$-th row of $A$ (that is the row right above the horizontal line) we subtract row $\frac{d}{2}-3$, add row $\frac{d}{2}-4$, subtract row $\frac{d}{2}-5$ and so on until we add or subtract row 1. We move this modified row below the horizontal line in the matrix and get:
\[ \begin{pNiceArray}{ccccc|ccccc}
    1 &&&&&0 &\\ 1 &1&&&&2 \\[-1ex] & 1& \ddots &&&& 1\\[-1ex] & & \ddots & 1 &&& & \ddots \\ &&& 1 &1&&&&1& \\\hline  &&&  &&\pm 2&\mp 1& \ldots&-1&1 \\  &&\Block{3-1}<\Huge>{0}&&&1&1 \\[-1.3ex] &&&&&&1&\ddots \\[-1ex] &&&&&&&\ddots&1\\&&&&&&&&1&2
\end{pNiceArray}.\]  
This block matrix is divided into square blocks. Its determinant is given by the determinant of the lower right block, which up to sign is the determinant of 
\[ \begin{pmatrix}
    2&-1&1 & \ldots & (-1)^{d/2} & (-1)^{d/2+1} \\ 1&1 &  & \\ &1&1\\ && \ddots & \ddots \\ &&& \ddots & 1 \\ &&&&1&2
\end{pmatrix}.\]
This matrix has size $\frac{d}{2}-2$, and we successively subtract column $i$ from column $i+1$ for $i = 1 \upto \frac{d}{2}-4$ to obtain
\[ \begin{pmatrix}
    2&-3&4 & \ldots & (-1)^{d/2} (\frac{d}{2}-2) & (-1)^{d/2+1} \\ 1&0 &  & \\ &1&0\\ && \ddots & \ddots \\ &&& \ddots & 0 \\ &&&&1&2
\end{pmatrix}.\]
The determinant of this matrix is $(-1)^{d/2} (2 (\frac{d}{2} - 2) + 1) = \pm (d-3).$ This is non-zero, because $d \leq n+2$ and $n \leq p = \charakt(\kk)$. So the matrix $A$ representing $\Delta_{d-1}: S_d' \to S_{d-1}$ has full rank, and hence $\Delta_{d-1} : S_d' \to S_{d-1}$ is an isomorphism.    
\end{proof}

       \begin{Proposition}\label{PropositionDegree3AndNeven}
      Let $n\ge 4$ be an even integer. There exists a homogeneous invariant $f_n \in \kk [V_{n}]^G$  of degree 3 such that $f_n = x^2_1 x_n + h$, where $h\in \kk [V_{n-1}]=\kk [x_1, x_2,\dots , x_{n-1}]$.
  \end{Proposition}
  
  \begin{proof}
We have $\Delta_{n+1}(x_1^2 x_n) = \Delta(x_1^2 x_n) = x_1^2 x_{n-1} \in S_{n+1}$. So by Lemma~\ref{LemmaDeg3andEven} with $d = n+2$ there exists $g_{n+2} \in S_{n+2}'$ such that $\Delta_{n+1}(g_{n+2}) = \Delta_{n+1}(x_1^2 x_n)$ and so $\Delta(x_1 x_n^2 
- g_{n+2}) \in S_{n-1} + S_n$ by \eqref{DeltaSd}. By Lemma~\ref{LemmaDeg3andOdd} with $d = n+1$, there exists $g_{n+1} \in S_{n+1}$ such that $\Delta_n(x_1 x_n^2)=\Delta_n( g_{n+2})$  and so $\Delta(x_1^2 x_n - g_{n+2} - g_{n+1}) \in S_{n-2} + S_{n-1}$ again by \eqref{DeltaSd}.

Continuing in this fashion, inductively, but in reverse order for $i = n + 2 \upto 4$ we find polynomials $g_i$ such that
\begin{itemize}
    \item $g_i \in S_i$ (and if $i$ is even and $i > 4$, then $g_i \in S_i'$),
    \item $\Delta(x_1^2 x_n - \sum_{j=i}^{n+2} g_j) \in S_{i-3} + S_{i-2}$.
\end{itemize}
In the end, define $h := - \sum_{j =4}^{n+2} g_j$. Then $\Delta(x_1^2 x_n + h) \in S_{1} + S_2 = \{0\}$, and hence $x_1^2 x_n + h$ is an invariant. In addition, since $h \in S_4 \oplus \ldots \oplus S_{n+1} \oplus S_{n+2}'$, we see that no monomial of $h$ is divisible by $x_n$, so $h \in \kk[x_1 \upto x_{n-1}]$.
\end{proof} 


  \begin{Def}\label{DefFNofDeg3}
As in Definition~\ref{DefFNofDeg2}, we can define $f_n$ (for $n\ge 4$ even) as the unique invariant produced by the proof of Proposition~\ref{PropositionDegree3AndNeven} if we add that $\Delta_{d-1}$ restricted to~$S_d'$ (in Lemma~\ref{LemmaDeg3andEven}) or to $\widehat{S_d}$ (in the proof of Lemma~\ref{LemmaDeg3andOdd}) is, in fact, an isomorphism. For $n = 4$ we get
\[ f_4 = x_1^2 x_4 - x_1 x_2 x_3 + \frac{1}{3} x_2^3  - \frac{1}{3} x_1^2 x_2.\]
  \end{Def}

  
  \begin{Theorem}\label{TheoremDegree3AndNeven}
      Let $n\ge 4$ be an even integer, and let $B_n$ be the open set defined in~\eqref{DefBn}. Furthermore, let
      $S \subseteq \kk [V_{n-1}]^G$ be a $B_{n-1}$-separating set. Then $S \cup \{f_n\}$ is a $B_n$-separating set, where $f_n$ is the homogeneous invariant   in $\kk [V_{n}]^G$ of degree~3  defined just before the theorem.
  \end{Theorem}

\begin{proof}
    The proof of Theorem \ref{TheoremDegree2andNodd} carries over verbatim, except that we now take the invariant $f_n$ from Proposition~\ref{PropositionDegree3AndNeven}. 
\end{proof}

Proposition~\ref{PropositionDegree3AndNeven} does not hold for $n = 2$   where the invariant ring is generated as a $\kk$-algebra by $x_1$ and $N(x_2) = x_2^p-x_1^{p-1}$, the \emph{norm} of $x_2$, and hence no invariant of degree 3 of that form exists.

\subsection{Main results}\label{SubSecMainResults}

\begin{Theorem}\label{MainTheoIndecomp}
Let $G$ be a cyclic group of prime order $p$, and $V = V_n$ an indecomposable representation of $G$ over a field $\kk$ of characteristic $p$ with $\dim(V) = n < \infty$. Then 
\begin{equation}\label{Invariants} x_1, \ N(x_2)=x_1^{p-1}x_2 - x_2^p, \ f_3 , \ \ldots , \ f_n \end{equation}
are generically separating and generate the invariant field $\kk(V)^G$, where $f_m$ is the  homogeneous invariant of degree 2 or 3  from Definition~ \ref{DefFNofDeg2} and~\ref{DefFNofDeg3}  for $3\le m\le n$. 
\end{Theorem}

\begin{proof}
We show that these invariants are separating on the Zariski open, $G$-stable subset $B_n$ defined in Equation \eqref{DefBn}.  For $n = 2$ the invariant ring is generated (as an $\kk$-algebra) by $x_1$ and $N(x_2)$, so they are both separating and generating for the field extension $\kk \subseteq \kk(V)^G$.  Assume that $m\ge 3$ and
$x_1, N(x_2),  f_3,  \ldots, \ f_{m-1}$ is $B_{m-1}$-separating. Then 
$x_1, N(x_2),  f_3,  \ldots, \ f_{m-1}, f_m$ is $B_m$-separating by Theorem~\ref{TheoremDegree2andNodd} if $m$ is odd or by Theorem \ref{TheoremDegree3AndNeven} if $m$ is even. Hence, the assertion of the theorem on generic separation follows by induction.

To show that the invariants in \eqref{Invariants} generate the field of invariants, we use a criterion from Campbell and Chuai~\cite{MR2334859}. Note that for $3\le m\le n$, the invariant $f_m$ lies in $\kk[x_1 \upto x_{m}]$ and, considered as a polynomial in $x_m$ with coefficients in $\kk[x_1 \upto x_{m-1}]$,  is of degree 1 in $x_m$, hence of the smallest positive degree in $x_m$. In addition, $N(x_2)$ is in $\kk[x_1,x_{2}]$ and  is of the smallest positive degree (considered as a polynomial in $x_2$) in~$x_2$. With this, the result follows from~\cite[Theorem 2.4]{MR2334859}.
\end{proof}

\begin{rem}
A generating set of polynomials for the invariant field for the case considered in the theorem is also given by \cite[Theorem 3.1]{MR2334859}. The generating set in this source contains (in addition to $x_1$ and $N(x_2)$)  homogeneous  polynomials that are called Shank invariants.  They were  were constructed in \cite[Theorem 2.3]{MR2334859} and are of the form $x_1^{m-2}x_m+h_m$ for  $3\le m\le n$, where  $h_m\in \kk[x_1 \upto x_{m-1}]$. This generating set is very similar to the generating set in the theorem but contains polynomials of each degree from 1 to $n-1$.
\end{rem}

\begin{rem}
The indecomposable representation of maximal dimension $n = p$ is isomorphic to the regular representation of~$G$. Hence, it follows from~\cite{edidin2025orbitrecoveryinvariantslow} that the invariant field $\kk(V_p)^G$ is generated by polynomial invariants of degree $\leq 3$. 

The invariant $N(x_2)$ in Theorem~\ref{MainTheoIndecomp} catches the eye as the only polynomial of degree higher than~3. So, it is tempting to try to reduce it, but the lower bound $\betafield(V,G) \geq \sqrt[n]{p}$ is proven in~\cite[Theorem 3.2]{BLUMSMITH2024107693}.  It follows that there does not exist a constant $c$ such that $\kk(V_n)^G$ is generated by polynomials of degree $\leq c$ for all primes $p$. So we cannot replace $N(x_2)$ with another polynomial whose degree does not depend on $p$.
\end{rem}

As any finite-dimensional representation~$V$ of~$G$ decomposes uniquely into indecomposable subrepresentations, we get the following result.

\begin{Theorem}\label{MainTheoDecomp}
Let $G$ be a cyclic group of prime order $p$, and $V$ representation of~$G$ over a field $\kk$ of characteristic $p$ with $\dim(V) = n < \infty$. Let~$r$ denote the number of trivial summands in $V$, and let~$m$ denote the number of nontrivial, indecomposable summands of $V$. Then there exist $n$ homogeneous and generically separating invariants consisting of
\begin{itemize}
\item $m+r = \dim(V^G)$ invariants of degree 1
\item $m$ invariants of degree $p$
\item $n-2m-r$ invariants of degree $2$ or $3$.
\end{itemize}
Furthermore, these $n$ invariants generate the invariant field.
\end{Theorem}

\begin{proof}
Let $V = \oplus _{i=1}^m U_i$ be a decomposition into indecomposable subrepresentations~$U_i$ with $\dim(U_i) = n_i$. Each $U_i$ is isomorphic to a $V_{n_i}$ for an integer $1 \leq n_i \leq p$. We can find bases of $U_i$ such that the dual basis of their union are the indeterminates
\[ x_{1,1} \upto x_{1,n_1},\, x_{2,1} \upto x_{2, n_2} \,  \upto  \, x_{m,n_m} \]
for the coordinate ring $\kk[V]$ and such that a generator $\sigma$ of $G$ acts on $x_{i,j}$ via
\[ \sigma (x_{i,1}) = x_{i,1}, \quad \sigma (x_{i,j}) = x_{i,j-1} + x_{i,j} \quad \text{ for } i = 1 \upto m, j = 2 \upto n_i.\]
For each $i$,  there are homogeneous invariants $f_{i,3} \upto f_{i,n_i} \in \kk[U_i]^G$ of degree 2 or 3 as in Theorem~\ref{MainTheoIndecomp}. As in the proof of that theorem, $f_{i,j}$ is of the smallest positive degree in $x_{i,j}$, hence it follows from~\cite[Theorem 2.4]{MR2334859} that they generate the invariant field. By Proposition~\ref{PropGenSeparatingAndFieldGen} they are also generically separating.

Alternatively, we could use the Zariski open, $G$-stable set \[ B = \{ (c_{1,1} \upto c_{m,n_m}) \in V \mid c_{i,1} \neq 0 \text{ for all } i = 1 \upto m \}, \] and a similar reasoning as in  Theorem \ref{TheoremDegree2andNodd} and \ref{TheoremDegree3AndNeven}, to see that these invariants are generically separating.
\end{proof}

In the end, we add another aspect to the invariants we found. For a fixed $n$ we can consider the operators $\sigma$ and $\Delta = \sigma - \id$ not only on $\kk[V_n]$, but by the same rules also on $\Z[x_1 \upto x_n]$:
\[ \sigma (x_i) = x_{i-1} + x_i  \quad \text{(for } 2\le i \le  n), \quad \text{ and } \sigma (x_1) = x_1.\]
This corresponds to looking at the action of a Jordan block with 1's along the diagonal in characteristic 0, so to the action of an infinite cyclic group $\Z$. There is a map from $\Z[x_1 \upto x_n]^\Z$ to $\kk[V_n]^G$ by reducing coefficients modulo $p$, and invariants in the image of this map are called \emph{integral invariants}.\footnote{In \cite{Shankv4v5} they are actually called \emph{rational invariants}, which we avoid for possible confusion with elements of the invariant field $\kk(V)^G$.} 

Inspecting the proofs of the four lemmas in~Subsections~\ref{SubSecDeg2} and~\ref{SubSecDeg3} again, we see that the invariants we found can be defined over~$\Z$ if we multiply them successively with the determinants of the matrices involved in their construction. These determinants are~$2$ in Lemma~\ref{LemmaDeg2andEven}, $d-3$ in Lemma~\ref{LemmaDeg3andEven}, and 1 elsewhere. It follows that some multiples of 
all $f_m$ for $3\le m\le n$ in Theorem \ref{MainTheoIndecomp} are integral invariants. 
For instance we have
\begin{align*} 2 f_3 &= 2 x_1 x_3 - x_2^2 + x_1 x_2,\\
3 f_4 &= 3 x_1^2 x_4 - 3 x_1 x_2 x_3 + x_2^3 - x_1^2 x_2.
\end{align*}

 \begin{rem}
Using graded lexicographic order with $x_n > x_{n-1} > \ldots > x_1$, we have constructed integral invariants with lead terms $x_1x_n$ or $x_1^2 x_n$, respectively.  Shank~\cite[Theorem 2.2]{Shankv4v5} showed that using graded reverse lexicographic order with $x_n > x_{n-1} > \ldots > x_1$ there is no integral invariant with lead term $x_1 x_n$ or $x_1^2 x_n$ or, in fact, $m x_n$ for any monomial $m$ in $x_1 \upto x_{n-1}$.
\end{rem}

\bibliographystyle{siam}
\bibliography{gensep}

\end{document}